\documentclass[11pt]{amsart}

\usepackage{amsmath}               
\usepackage{amsfonts}       
\usepackage{amsthm}                
\usepackage{amssymb}
\usepackage{verbatim}

\numberwithin{equation}{section}

\theoremstyle{plain}
\newtheorem{thm}[equation]{Theorem}
\newtheorem{lemma}[equation]{Lemma}

\theoremstyle{definition}
\newtheorem{remark}[equation]{Remark}

\DeclareMathOperator{\charp}{char}

\renewcommand{\bar}[1]{#1\llap{$\overline{\phantom{\rm#1}}$}}

\title{Factorizations of certain bivariate polynomials}

\author{Michael E. Zieve}
\address{
Michael Zieve \\
  Department of Mathematics \\
  University of Michigan \\
  Ann Arbor, MI 48109--1043 \\
  USA
}
\address{Mathematical Sciences Center \\
Tsinghua University \\
Beijing 100084 \\
 China}
\email{zieve@umich.edu}
\urladdr{www.math.lsa.umich.edu/$\sim$zieve/}

\date{}

\begin{document}

\begin{abstract}
We determine the factorization of $X f(X) - Y g(Y)$ over $K[X,Y]$ for all squarefree additive
polynomials $f,g\in K[X]$ and all fields $K$ of odd characteristic.  This answers a question
of Kaloyan Slavov, who needed these factorizations in connection with an algebraic-geometric analogue of
the Kakeya problem.
\end{abstract}

\thanks{The author thanks the NSF for support under grant DMS-1162181.}

\maketitle


\section{Introduction}

Many authors have studied the problem of determining the factorizations of all members of some infinite
class of bivariate polynomials.  Special attention has been paid to the case of polynomials with
separated variables, that is, polynomials of the form $F(X)-G(Y)$ where $F(X)$ and $G(Y)$  are
univariate polynomials.  Although there have been some remarkable advances, still there is no general
solution.

We will determine the factorizations of $F(X)-G(Y)$ for a special class of polynomials $F$ and $G$.
These particular factorizations are needed in order to prove an extreme case of a conjecture by
Kaloyan Slavov, which arose in connection with an algebraic-geometric analogue of the Kakeya problem.

We now define the polynomials under consideration.
Let $K$ be a field of characteristic $p>0$.  A polynomial $f(X)\in K[X]$ is \emph{additive} if it satisfies
the identity $f(X+Y)=f(X)+f(Y)$.  As is well-known,  additive polynomials are precisely the polynomials of the form
$f(X)=\sum_{i=0}^m \alpha_i X^{p^i}$ where $\alpha_i\in K$; for instance, see \cite[Cor.~1.1.6]{G96}.
In this note we consider polynomials of the form $X f(X) - Y g(Y)$, where $f(X)$ and $g(X)$ are squarefree
additive polynomials in $K[X]$ with $p>2$.  Note that since $f(X)$ is additive we have $f'(X)=f'(0)$, so that
$f(X)$ is squarefree if and only if $f'(0)\ne 0$.  We prove the following result.

\begin{thm} \label{zmain}
Let $K$ be a field of characteristic $p>2$, let $f,g\in K[X]$ be additive polynomials with $f'(0)g'(0)\ne 0$,
and let $F(X):=X f(X)$ and $G(X):=X g(X)$.
\begin{enumerate}
\item \label{zmain1} $F(X)-G(Y)$ is reducible in $K[X,Y]$ if and only if $g(Y)=\delta\cdot f(\delta Y)$ for some
$\delta\in\bar K^*$ such that either $\deg(f)>1$ or $\delta\in K$.
\item \label{zmain2} If $g(Y) = \delta \cdot f(\delta Y)$ for some $\delta\in\bar K^*$, and $\deg(g)>1$, then
$\delta^2\in K$ and
$F(X)-G(Y)$ is the product of $(X-\delta Y)$, $(X+\delta Y)$, and an irreducible polynomial in $K[X,Y]$.
\end{enumerate}
\end{thm}

\begin{remark}
For completeness, we observe that if $f(X)=\alpha X$ and $g(Y)=\beta Y$ then $X f(X)-Y g(Y)=
\alpha(X-\gamma Y)(X+\gamma Y)$ where $\gamma^2=\beta/\alpha$.
\end{remark}


\section{Preliminary results}

In this section we state two results which will be used in our proof of Theorem~\ref{zmain}.

\begin{lemma} \label{zfried}
Let $K$ be a field, and let $F,G\in K[X]$ have nonzero derivatives.  Then there exist $F_1,F_2,G_1,G_2\in K[X]$
such that all of the following hold:
\begin{itemize}
\item $F=F_1\circ F_2$ \quad and \quad $G=G_1\circ G_2$
\item the factors of $F(X)-G(Y)$ in $K[X,Y]$ are precisely the polynomials $H(F_2(X),G_2(Y))$ where
$H(X,Y)\in K[X,Y]$ is a factor of $F_1(X)-G_1(Y)$
\item the splitting field of $F_1(X)-t$ over $K(t)$ equals the splitting field of $G_1(X)-t$ over $K(t)$,
where $t$ is transcendental over $K$.
\end{itemize}
\end{lemma}

Fried proved a version of this result in \cite[Prop.~2]{Fried}; see also \cite[Thm.~8.1]{BT}
for a nice exposition of Fried's proof, which yields the above statement.  Although these references state
the result for fields of characteristic zero, the proofs extend at once to fields of arbitrary characteristic.
A different proof of Lemma~\ref{zfried} is given in \cite[Thm.~1.1]{EK}, and yet another proof
will appear in the forthcoming paper \cite{MZ}.

We also use the following result about polynomial decomposition:

\begin{lemma} \label{zunique}
Let $K$ be a field, and let $F,G,\hat F,\hat G\in K[X]$ be nonconstant polynomials satisfying
\begin{itemize}
\item $F\circ G=\hat F\circ \hat G$
\item $\deg(F)=\deg(\hat F)$
\item $\charp(K)\nmid\deg(F)$.
\end{itemize}
Then $\hat G = L\circ G$ for some degree-one polynomial $L(X)\in K[X]$.
\end{lemma}

This result was proved for $\charp(K)=0$ in \cite[\S 2]{Levi}.  The proof extends at once
to arbitrary characteristic, as noted in \cite[Prop.~2.2]{Turnwald}.


\section{The main result}

In this section we prove Theorem~\ref{zmain}.  We begin with two lemmas.

\begin{lemma} \label{zram}
Let $K$ be a an algebraically closed field of odd characteristic, and let $F(X):=X f(X)$ where
$f(X)\in K[X]$ is an additive polynomial with $f'(0)\ne 0$.  Write $f(X)+Xf'(0)=X\hat{f}(X^2)$
with $\hat{f}\in K[X]$.  Then $\hat{f}$ is a squarefree polynomial of degree $(\deg(f)-1)/2$
with $\hat{f}(0)\ne 0$, and the nonzero finite critical values of $F(X)$ are the values
$-\beta \hat{f}(0)$ where $\beta$ varies over the roots of $\hat{f}(X)$.
For each such value $\gamma$, the polynomial $F(X)-\gamma$ has two roots of multiplicity $2$
and all other roots of multiplicity $1$.
Also $F(X)$ has one root of multiplicity $2$, and all other roots of multiplicity $1$.
If we write $F=A\circ X^2$ with $A\in K[X]$, then the finite critical values of $A(X)$ are precisely
the nonzero finite critical values of $F(X)$, and for each such value $\gamma$ the polynomial
$A(X)-\gamma$ has one root of multiplicity $2$ and all other roots of multiplicity $1$.
\end{lemma}

\begin{proof}
Since $f'(X)=f'(0)$, we have $F'(X)=f(X) + X f'(0)$.  This is an additive polynomial which is squarefree
because $F''(X)=2f'(0)\ne 0$.  Since $F''(X)$ is a nonzero constant, it follows that all roots of $F(X)-\gamma$
have multiplicity at most $2$ for every $\gamma\in K$.  Moreover, if
$F'(\alpha)=0$ then $F(\alpha)=\alpha f(\alpha)=-\alpha^2 f'(0)$.  Since $F'(X)=2XA'(X^2)$,
the result follows.
\end{proof}

\begin{lemma} \label{zdecomp}
Let $K$ be a field of odd characteristic, and let $f(X)\in K[X]$ be an additive polynomial with $f'(0)\ne 0$.
If $G,H\in K[X]$ satisfy $G\circ H=X f(X)$ then either $\deg(G)=1$ or $\deg(H)=1$ or
$H=L\circ X^2$ for some degree-one $L(X)\in K[X]$.
\end{lemma}

\begin{proof}
Write $F(X):=X f(X)$, and suppose that $G\circ H=F$ where $\deg(G)>1$.
Since $\charp(K)\nmid\deg(G)$, the derivative $G'(X)$ has degree $\deg(G)-1\ge 1$ and hence has
a root $\alpha$.  Thus $(X-\alpha)^2$ divides $G(X)-G(\alpha)$, so substituting $H(X)$ for $X$ shows
that $(H(X)-\alpha)^2$ divides $F(X)-G(\alpha)$.  Now Lemma~\ref{zram} implies that $H(X)-\alpha$
has degree at most $2$, so that $\deg(H)\le 2$.  Since we know that $F=A\circ X^2$ for
some $A\in K[X]$, Lemma~\ref{zunique} implies that if $\deg(H)=2$ then $H=L\circ X^2$ for some
degree-one $L(X)\in K[X]$.
\end{proof}

Now we prove the first part of Theorem~\ref{zmain}.

\begin{proof}[Proof of part (\ref{zmain1}) of Theorem~\ref{zmain}]
First we prove the ``if" implication.  Suppose that $g(Y)=\delta \cdot f(\delta Y)$ for some $\delta\in\bar K^*$
Comparing coefficients of $Y$ shows that $g'(0)=\delta^2 f'(0)$, so that $\delta^2\in K^*$.
Note that all terms of $F(X)$ have even degree, so that $F=A\circ X^2$ for some $A\in K[X]$.
Now $F(X)-G(Y)=F(X)-F(\delta Y)=A(X^2)-A(\delta^2 Y^2)$ is divisible by $X^2-\delta^2 Y^2$ in
$K[X,Y]$, and hence is reducible if either $\deg(F)>2$ or $\delta\in K$.

Next we prove the ``only if" implication.  Assume that $F(X)-G(Y)$ is reducible in $K[X,Y]$.
Let $F_1,F_2,G_1,G_2\in K[X]$ satisfy the conclusions of Lemma~\ref{zfried}. Reducibility of $F(X)-G(Y)$
implies that $F_1(X)-G_1(Y)$ is reducible, so that $\deg(F_1)>1$.  Let $t$ be transcendental over $K$,
and let $\Omega$ be the splitting field of $F_1(X)-t$ over $K(t)$, which is also the splitting field of $G_1(X)-t$
over $K(t)$.  Let $u,v\in\Omega$ satisfy $F_1(u)=t=G_1(v)$.  Then $t=\infty$ is totally ramified in
$K(u)/K(t)$, so since $\charp(K)\nmid\deg(F_1)$ it follows (e.g.\ by Abhyankar's lemma)
that every place of $\Omega$ which lies over
$t=\infty$ must have ramification index equal to $\deg(F_1)$ in $\Omega/K(t)$.  Likewise,
each such place also has ramification index $\deg(G_1)$, so that $\deg(F_1)=\deg(G_1)$.
Write $F=A\circ X^2$ and $G=B\circ X^2$ with $A,B\in K[X]$, and also
$f(X)+Xf'(0)=X\hat{f}(X^2)$ and $g(X)+Xg'(0)=X\hat{g}(X^2)$ with $\hat{f},\hat{g}\in K[X]$.
Since $\deg(F_1)>1$, Lemma~\ref{zdecomp} implies that there is a degree-one $L\in K[X]$ for which
$F_1$ is either $A\circ L$ or $F\circ L$, so by Lemma~\ref{zram} the nonzero finite critical values of $F_1$
are the values $-\alpha \hat{f}(0)$ where $\alpha$ is a root of $\hat{f}(X)$.
Since $\Omega/K(t)$ is the Galois closure of $K(x)/K(t)$, 
a place $P$ of $K(t)$ lies under a place of $K(x)$ which is ramified in $K(x)/K(t)$
if and only if $P$ lies under a place of $\Omega$ which is ramified in $\Omega/K(t)$.
It follows that the critical values of $F_1$ are identical to the critical values of $G_1$, so
the values
$\alpha\hat{f}(0)$ where $\hat{f}(\alpha)=0$ are the same as the values $\beta\hat{g}(0)$ where
$\hat{g}(\beta)=0$.  Writing $\gamma:=\hat{g}(0)/\hat{f}(0)$, this implies that the roots of $\hat{g}(X)$
are the values $\alpha/\gamma$ where $\hat{f}(\alpha)=0$, which in turn are the roots of
$\hat{f}(\gamma X)$.  Since $\hat{f}$ and $\hat{g}$ are squarefree (by Lemma~\ref{zram}), 
it follows that $\hat{g}(X)$ is a constant multiple of $\hat{f}(\gamma X)$, and substituting $X=0$
shows that the constant is $\gamma$.  Thus $\hat{g}(X)=\gamma\cdot\hat{f}(\gamma X)$, so that
\[
g(X)+Xg'(0)=X\hat{g}(X^2)=\gamma X\hat{f}(\gamma X^2) =\delta(f(\delta X)+\delta Xf'(0))
\]
where $\delta^2=\gamma$.  Equating coefficients of $X$ on both sides yields
$2g'(0)=\delta(2\delta f'(0))$, so subtracting $X g'(0)$ from both sides yields $g(X)=\delta f(\delta X)$,
as desired.
\end{proof}

Finally, we prove the second part of Theorem~\ref{zmain}.

\begin{proof}[Proof of part (\ref{zmain2}) of Theorem~\ref{zmain}]
As in the first part of the previous proof, we write $F=A\circ X^2$ with $A\in K[X]$ and note that
$F(X)-G(Y)=A(X^2)-A(\delta^2 Y^2)$ is divisible by $X^2-\delta^2 Y^2$ in $K[X,Y]$.
Writing $B(X,Y):=(A(X)-A(Y))/(X-Y)$, the ratio $(F(X)-G(Y))/(X^2-\delta^2 Y^2)$ equals
$B(X^2,\delta^2 Y^2)$.  It suffices to show that $B(X^2,\delta^2 Y^2)$ is irreducible in $\bar{K}[X,Y]$,
so in what follows we assume that $K$ is algebraically closed.
Lemma~\ref{zram} shows that the derivative $A'(X)$ has $\deg(A)-1$ distinct roots, all of which
are simple, and distinct roots have distinct images under $A$.  Thus $A(X)$ is a Morse function in the
sense of \cite[p.~39]{Serre},  so if $t$ is transcendental over $K$ then the Galois group
of $A(X)-t$ over $K(t)$ is the symmetric group on $\deg(A)$ letters \cite[Thm.~4.4.5]{Serre}.
In particular, if $\deg(A)>1$ then this group is doubly transitive, so that $B(X,Y)$ is irreducible
 \cite[Lemma~3.2]{Turnwald}.
Let $x$ satisfy $F(x)=t$, and put $u:=x^2$ so that $A(u)=t$.  Let $H(X,Y)$ be an irreducible factor
of $B(X^2,\delta^2 Y^2)$, let $y$ satisfy $B(x^2,\delta^2 y^2)=0$, and put $v:=\delta^2 y^2$ so that
$A(u)=A(v)$.
Irreducibility of  $B(X,Y)$  implies that  $B(u,Y)$  is the minimal polynomial for
$v$  over  $K(u)$,  whence  $[K(u,v):K(u)]=\deg(A)-1$.  Also, since $t=\infty$ is totally ramified
in both  $K(u)/K(t)$  and  $K(v)/K(t)$,  and in both cases its ramification index is  $\deg(A)$  which
is coprime to  $p$,  Abhyankar's lemma implies that every place of  $K(u,v)$  which
lies over  $t=\infty$  has ramification index in  $K(u,v)/K(t)$  equal to  $\deg(A)$,  so each such
place is unramified in  $K(u,v)/K(u)$.  Thus  $u=\infty$  is unramified in  $K(u,v)/K(u)$,  but it is
totally ramified in  $K(x)/K(u)$,  so we must have $[K(x,v):K(u,v)]=2$.  Since  $A(X)$  is squarefree
and  $A(0)=0$,  we know that  $K(v)/K(t)$  is unramified over  $t=0$,  so that  $K(u,v)/K(u)$  is
unramified over  $u=0$.  Moreover, the places of  $K(u,v)$  which lie over  $u=0$  have $v$-values
being all the nonzero roots of  $A$.  Since  $K(u,v)/K(u)$  has branch points  $u=0$  and  $u=\infty$,
both of which are unramified in  $K(u,v)/K(u)$,  it follows that the branch points of
$K(x,v)/K(u,v)$  are the places lying over either $u=0$ or $u=\infty$.  Likewise the branch points of
$K(u,y)/K(u,v)$  are the places lying over either $v=0$ or $v=\infty$.  But the places of $K(u,v)$
which lie over $u=0$ do not lie over $v=0$, so  $K(x,v)/K(u,v)$  and  $K(u,y)/K(u,v)$  have distinct
branch points and hence are distinct extensions, whence $[K(x,y):K(x,v)]=2$.  This shows that
$[K(x,y):K(x)]=2[K(x,v):K(x)]=2(\deg(A)-1)$,  so that the $Y$-degree of  $H(X,Y)$  is  $2(\deg(A)-1)$,
which equals the $Y$-degree of  $B(X^2,\delta^2 Y^2)$.  The same argument applies to $X$-degrees, so
that  $B(X^2,\delta^2 Y^2)$  is irreducible.
\end{proof}


\end{document}